\documentclass{amsart}
\usepackage{amssymb, amsmath, latexsym, mathabx, dsfont}

\renewcommand{\baselinestretch}{\baselinestretch}
\renewcommand{\baselinestretch}{1.1}
\numberwithin{equation}{section}
\usepackage{mathrsfs}

\newtheorem{thm}{Theorem}[section]
\newtheorem{lem}[thm]{Lemma}
\newtheorem{cor}[thm]{Corollary}
\newtheorem{prop}[thm]{Proposition}
\newtheorem{defn}[thm]{Definition}
\newtheorem{rmk}[thm]{Remark}

\newcommand{\ra}{{\ \rightarrow\ }}

\newcommand{\gen}{\text{gen}}

\newcommand{\ord}{\text{ord}}

\newcommand{\z}{{\mathbb Z}}

\newcommand{\n}{{\mathbb N}}

\newcommand{\MM}[1]{\mathcal{M}(#1)}

\newcommand{\Mod}[1]{\ (\mathrm{mod}\ #1)}

\begin{document}
\title[Minimal universality criterion sets] {Minimal universality criterion sets on the representations of quadratic forms}

\author{Kyoungmin Kim, Jeongwon Lee, and Byeong-Kweon Oh}

\address{Department of Mathematics, Sungkyunkwan University, Suwon 16419, korea}
\email{kiny30@skku.edu}
\thanks{This work of the first author was supported by the National Research Foundation of Korea(NRF) grant funded by the Korea government(MSIT) (NRF-2020R1I1A1A01055225)}

\address{Mechatronics R\&D Center, Samsung Electronics, Hwaseong 18848, Korea}
\email{jeong.1.lee@samsung.com}

\address{Department of Mathematical Sciences and Research Institute of Mathematics, Seoul National University, Seoul 08826, Korea}
\email{bkoh@snu.ac.kr}
\thanks{This work of the third author was supported by the National Research Foundation of Korea(NRF) grant funded by the Korea government(MSIT) (NRF-2019R1A2C1086347 and NRF-2020R1A5A1016126).}



\begin{abstract} 
For a set $S$ of (positive definite and integral) quadratic forms with bounded rank, a quadratic form $f$ is called $S$-universal if it represents all quadratic forms in $S$. A subset $S_0$ of $S$ is called an $S$-universality criterion set if any $S_0$-universal quadratic form is $S$-universal. We say $S_0$ is minimal if there does not exist a proper subset of $S_0$ that is an $S$-universality criterion set.     In this article, we study various properties of minimal universality criterion sets. 
In particular,  we show that  for `most' binary quadratic forms $f$, minimal $S$-universality criterion sets are unique in the case when $S$ is the set of all subforms of the binary form $f$. 
\end{abstract}
\maketitle

\section{Introduction}
Let $S$ be a set of (positive definite integral) quadratic forms with bounded rank. A  quadratic form $f$ is called {\it $S$-universal} if it represents all quadratic forms in the set $S$. A subset $S_0$ of $S$ is called {\it an $S$-universality criterion set}  if any $S_0$-universal quadratic form  is $S$-universal. Conway and Schneeberger's 15-theorem \cite{c} says that the set $\{1,2,3,5,6,7,10,14,15\}$ is an $S$-universality criterion set, where $S$ is the set of all positive integers (see also \cite{b}). Note that any positive integer $a$ corresponds to the unary quadratic form $ax^2$. For an arbitrary set $S$ of quadratic forms with bounded rank, the existence of a finite $S$-universality criterion set was proved by the third author and his collaborators in \cite{kko}. 

An $S$-universality criterion set $S_0$ is called {\it minimal} if any proper subset of $S_0$ is not an $S$-universality criterion set. In \cite{kko}, the authors proposed the following two questions on the minimal universality criterion sets: Let $\Gamma(S)$ be the set of all $S$-universality criterion sets.
\begin{itemize}
\item [(i)] For which $S$ is there a unique minimal $S_0 \in \Gamma(S)$?
\item [(ii)] Is $\vert S_0\vert=\gamma(S)$ for every minimal $S_0 \in \Gamma(S)$? If not, when?
\end{itemize}
For the question (i), when $S$ is the set of all quadratic forms of rank $k$, the uniqueness of the minimal $S$-universality criterion set was proved by Bhargava \cite{b} for the case when $k=1$, and by Kominers 
\cite{kom1}, \cite{kom2} for the cases when $k=2$, $k=8$, respectively (see also \cite{k}, \cite{kko0}, and \cite{o1}). 

For the question (ii), Elkies, Kane, and Kominers \cite{ekk} answered  in the negative for some special set $S$ of quadratic forms.  In fact, they considered the set $S_f$ of all subforms of the quadratic form $f(x,y,z)= x^2+y^2+2z^2$. Clearly, the set $\{f\}$ itself is a minimal $S_f$-universality criterion set. They proved that any quadratic form that represents both subforms $x^2+y^2+8z^2$ and $2x^2+2y^2+2z^2$ of $f$ also represents $f$ itself.  Therefore the set $\{x^2+y^2+8z^2, 2x^2+2y^2+2z^2\}$ is also a minimal $S_f$-universality criterion set. 

In this article, we prove that  the question (i) is true if $S$ is any set of positive integers, that is, any set of unary quadratic forms.  We also prove that minimal $\Phi_n$-universality criterion sets are not unique when $\Phi_n$ is the set of all quadratic forms of rank $n$ for any $n\ge 9$. To analyze the example given by Elkies, Kane, and Kominers more closely, we introduce the notion of a {\it recoverable} quadratic form.   A quadratic form $f$ is called recoverable if  minimal $S_f$-universality criterion sets are not unique in the case when $S_f$ is the set of all subforms of $f$.  The third author and his collaborators proved in \cite{jko} that any unary quadratic form is not recoverable.  In this article, we show that `most' binary quadratic forms are not recoverable, and in fact, there are infinitely many recoverable binary quadratic forms up to isometry.

The subsequent discussion will be conducted in the better adapted geometric language of quadratic spaces and lattices.  Throughout this article, we always assume that every $\z$-lattice $L=\z x_1+\z x_2+\dots+\z x_k$ is {\it positive definite and integral}, that is,  the corresponding symmetric matrix 
$$
M_L=(B(x_i,x_j)) \in M_{k\times k}(\z)
$$ 
is positive definite and  the scale $\mathfrak s(L)$ of the $\z$-lattice $L$ is $\z$. The corresponding quadratic map $Q: L \to \z$  will be defined by $Q(x)=B(x,x)$ for any $x\in L$.  For any positive integer $a$,  the $\z$-lattice obtained from $L$ by scaling $a$ will be denoted by $L^a$. Hence we have $M_{L^a}=a\cdot M_L$. 
The discriminant $dL$ of the $\z$-lattice $L$ will be defined by $dL=\det(M_L)$. We call $L$ is diagonal if $B(x_i,x_j)=0$ for any $i$ and $j$ with $i\ne j$. If $L$ is diagonal, then we simply write 
$$
L=\langle Q(x_1),\dots,Q(x_k)\rangle.
$$  
We say $L$ is even if $Q(x)$ is even for any $x \in L$. 

 If an integer $n$ is represented by $L$ over $\z_p$ for any prime $p$ including infinite prime, then we say that $n$ is represented by the genus of $L$, and we write $n  \ra  \gen(L)$.  Note that $n$ is represented by the genus of $L$ if and only if $n$ is represented by a $\z$-lattice in the genus of $L$. 
When $n$ is represented by the $\z$-lattice $L$ itself, then we write $n\ra L$.   We define
$$
Q(L)=\{ n \in \z : n \ra L\}.
$$
For any positive integer $n$, The cubic $\z$-lattice $I_n=\z e_1+\z e_2+\dots+\z e_n$ is the $\z$-lattice satisfying  $B(e_i,e_j)=\delta_{ij}$.

Any unexplained notation and terminology can be found in \cite{ki} or  \cite{om}.

\section{Uniqueness of the minimal universality criterion set}
 In  general, minimal $S$-universality criterion sets are not unique for an arbitrary set $S$ of quadratic forms with bounded rank.   
 In this section, we prove that the minimal $S$-universality criterion set is unique if $S$ is any subset of positive integers. 

Let $\mathbb N$ be the set of positive integers. For a positive integer $k$ and a nonnegative integer $\alpha$, we define the set of arithmetic progressions 
$$
A_{k,\alpha}=\{kn+\alpha : n\in \mathbb N \cup \{0\}\}.
$$  
If a $\z$-lattice $L$ represents all elements in $A_{k,\alpha}$, we simply write $A_{k,\alpha} \ra L$.

\begin{prop} \label{keyp}
Let $S=\{ s_0,s_1,s_2,\dots\}$ be a subset of $\n$ such that $s_i < s_{i+1}$ for any nonnegative integer $i$, and let $k$ be a positive integer. If there is a $\z$-lattice $\ell$  such that 
$$
s_0,s_1,\dots,s_{k-1} \in Q(\ell) \quad \text{and} \quad s_k \not \in Q(\ell),
$$
then there is a $\z$-lattice $L$ such that $Q(L)\cap S=S-\{s_k\}$.
\end{prop}
\begin{proof} 
First, we define
$$
\mathfrak C=\{ 0 \le u \le s_{k+1}-1 : A_{s_{k+1},u} \cap \{s_{k+1},s_{k+2},\dots\} \ne \emptyset\}=\{ c_1,c_2,\dots,c_v\},
$$
and for each $c \in \mathfrak C$, let $s(c)=\min(A_{s_{k+1},c} \cap \{s_{k+1},s_{k+2},\dots\})$. 
We define 
$$
L=\ell \perp s_{k+1} I_4 \perp \langle s(c_1),s(c_2),\cdots,s(c_v) \rangle.
$$
Since $s_{k+1} >s_k$ and $s(c_j) > s_k$ for any $j=1,2,\dots,v$, we see that $s_k$ is not represented by $L$. Furthermore, for any integer 
$a \in \{s_{k+1},s_{k+2},\dots\}$, there is a nonnegative integer $M$ and an integer $j$ with $1\le j \le v$ such that $a=s_{k+1}M+s(c_j)$. Since $M$ is represented by $I_4$,  the integer $a$ is represented by $L$. The proposition follows directly from this.  
\end{proof}

\begin{thm}  
For any set $S=\{s_0,s_1,s_2,\dots\}$ of positive integers, the minimal $S$-universality criterion set is unique.  
\end{thm}

\begin{proof} 
Without loss of generality, we may assume that $s_i < s_{i+1}$ for any nonnegative integer $i$. A positive integer $s_i \in S$ is called a truant of $S$ if there is a $\z$-lattice $L$ such that $L$ represents all integers in the set $\{s_0,s_1,\dots,s_{i-1}\}$, whereas $L$ does not represent $s_i$. Clearly, $s_0$ is a truant of $S$. Let $T(S)$ be the set of truants of $S$. Then, by Proposition \ref{keyp}, any $S$-universality criterion set should contain $T(S)$. Hence it suffices to show that $T(S)$ itself is an $S$-universality criterion set. Let $L$ be a $\z$-lattice that represents all integers in $T(S)$. Suppose that $L$ is not $S$-universal. Let $m$ be the smallest integer such that $s_m$ is not represented by $L$. Then, clearly, $s_m$ is a truant of $S$, and hence $s_m \in T(S)$. This is a contradiction. Therefore, $T(S)$ is the unique minimal $S$-universality criterion set.    
\end{proof}

\begin{rmk} {\rm As pointed out by \cite{h},  Bhargava also proved the above result.  However,  no proof of this has appeared in the literature to the author's knowledge.}  
\end{rmk}

Let $L$ be a $\z$-lattice of rank $m$. For any positive integer $j$ less than or equal to  $m$, the $j$-th successive minimum of $L$ will be denoted by $\mu_j(L)$ ( for the definition of the successive minimum, see Chapter 12 of \cite{ca}). It is well-known that there is a constant $\gamma_m$, which is called the Hermite constant, such that  $$
dL\le  \mu_1(L)\mu_2(L)\cdots \mu_m(L) \le \gamma_m^mdL
$$ 
(for the proof, see Proposition 2.3 of \cite{ea}). We define $\min(L)=\min\{ Q(x)\mid x \in L-\{0\}\}$. Note that $\min(L)=\mu_1(L)$.

\begin{thm}
For any positive integer $k$, there is a subset $S$ of positive integers such that the cardinality of its minimal universality criterion set is exactly $k$.
\end{thm}
\begin{proof}
Let $L= \z x_{1} + \z x_{2} + \cdots + \z x_{k}$ be a $\z$-lattice such that $Q(x_{i}) = k+i$ for any $i$ with $1 \le i \le k$. 
If $m$ is  the rank of $L$, then we have $m \le k$ and $\mu_{m}(L) \le 2k$.
It follows from $\mu_{1}(L) \le \mu_{2}(L) \le \dots \le \mu_{m}(L)$ that
$$
dL \le \mu_{1}(L) \mu_{2}(L) \cdots \mu_{m}(L) \le (2k)^{k}.
$$
Hence there are only finitely many candidates for $L$ up to isometry since the discriminant and the rank of $L$ are bounded.
Let $\{ L_{1}, L_{2}, \dots, L_{t} \}$ be the set of all possible candidates for $L$. We define
$$
S = \cap_{i=1}^{t} Q(L_{i}).
$$
Then from the definition of $S$, it is obvious that $\{k+1, k+2, \dots, 2k\}$ is an $S$-universality criterion set.

Put $M_{1} = \langle k+2 \rangle$.
Since $k+1$ is not represented by $M_{1}$,  there is a $\z$-lattice $N_{1}$ such that $Q(N_{1}) \cap S = S-\{k+1\}$  by Proposition \ref{keyp}.
For each $i=2,3, \dots, k$, put
$$
M_{i} = \langle k+1, \dots, k+i-1 \rangle.
$$
Then, one may easily show that $k+j  \ra  M_{i}$ for any $j = 1, \dots, i-1$, whereas $k+i$ is not represented by $M_{i}$.
Then, by Proposition \ref{keyp} again, there is a $\z$-lattice $N_{i}$ such that $Q(N_{i}) \cap S = S-\{k+i\}$.
This implies that $\{k+1, k+2, \dots, 2k\}$ is the minimal $S$-universality criterion set.
\end{proof}

Let $\Phi_n$ be the set of all $\z$-lattices of rank $n$. As explained in the introduction, it is known that the minimal $\Phi_n$-universality criterion set is unique if $n=1,2$, or $8$. For all the other positive integers $n$, the explicit minimal $\Phi_n$-universality criterion set is not known yet.

\begin{prop} 
For any integer $n$ with $n \ge 9$, there are infinitely many minimal $\Phi_n$-universality criterion sets.
\end{prop}

\begin{proof}  Let $\Phi_n^0=\{L_1,L_2,\dots,L_s\}$ be a minimal $\Phi_n$-universality criterion set. Assume that $L_i=I_{k_i} \perp \ell_i$, where $\min(\ell_i)\ge 2$.
If $n_0=\max\{k_i\}<n$, then $I_{n_0} \perp \ell_1\perp\dots\perp\ell_s$   represents all $\z$-lattices in $\Phi_n^0$, but it does not represent $I_n$. This is a contradiction. Therefore $n_0=n$, that is, $I_n \in \Phi_n^0$. Similarly, one may easily show that  there is an integer $j$ such that $L_j$ represents $D_m[1]$ for some integer $m \equiv 0 \Mod 4$ with $n-4\le m<n$.  Note that $L_j=D_m[1]\perp M$ for some $\z$-lattice $M$ with rank less than or equal to $4$. Without loss of generality, assume that $L_1=I_n$ and $L_2=D_m[1]\perp M$. Note that any $\z$-lattice that represents both $L_1$ and $L_2$ should represent $I_n\perp D_m[1]$. Furthermore, since  $I_n$ is $4$-universal, $L_j$ cannot represent $D_m[1]$ for any $j\ge 3$.   
Now we show that for any $\z$-lattice $N$ with rank $n-m$,
$$
\Phi_n^0(N)=\{I_n, D_m[1]\perp N,L_3,\dots, L_s\}
$$
is also a minimal $\Phi_n$-universality criterion set.  Assume that a $\z$-lattice $\mathcal L$ represents all $\z$-lattices in $\Phi_n^0(N)$. 
Since $I_n\perp D_m[1]$ is represented by $\mathcal L$, $L_2=D_m[1]\perp M$ is also represented by $\mathcal L$. Therefore, $\mathcal L$ is $n$-universal from the assumption that $\Phi_n^0$ is a  $\Phi_n$-universality criterion set. By using similar argument, one may easily show that $\Phi_n^0(N)$ is, in fact,  minimal. 
\end{proof}

\begin{rmk}{\rm
We conjecture that 
$$
\left\lbrace I_4, A_4, A_2\perp A_2, A_2\perp \begin{pmatrix} 2&1\\1&3\end{pmatrix}\right\rbrace
$$ 
is the unique minimal $\Phi_{4}$-universality criterion set. Here, $A_m=\z(e_1-e_2)+\z(e_2-e_3)+\dots+\z(e_m-e_{m+1})$ is a root $\z$-lattice, where $\{e_i\}$ is the standard orthonormal basis of the cubic $\z$-lattice $I_n$.}
\end{rmk}

\section{Recoverable $\z$-lattices}
In this section, we introduce the notion of {\it recoverable} $\mathbb{Z}$-lattices and give some properties on those $\mathbb{Z}$-lattices, and we show some necessary conditions and some sufficient conditions for $\z$-lattices to be recoverable.

In \cite{ekk}, Elkies and his collaborators gave an example of a set $S$ of ternary $\z$-lattices such that the sizes of minimal $S$-universality criterion sets vary. To explain their example more precisely, let $T$ be the set of all ternary sublattices of $\langle 1,1,2\rangle$. Then, clearly, $T_0=\{\langle1,1,2\rangle\}$ is a minimal $T$-universality criterion set. 
Furthermore, they proved that 
$$
T_1=\{ \langle 1,1,8\rangle, \langle 2,2,2\rangle\}
$$
is also a minimal $T$-universality criterion set. The point is that any $\z$-lattice that represents both $\langle 1,1,8\rangle$ and $\langle 2,2,2\rangle$, which are  sublattices of $\langle1,1,2\rangle$, also represents $\langle 1,1,2\rangle$ itself. 
From this point of view, the following definition seems to be quite natural:

\begin{defn}  
Let $\ell$ be a $\z$-lattice and let $S_0=\{\ell_1,\ell_2,\dots,\ell_t\}$ be the set of proper $\z$-sublattices of $\ell$. We say $\ell$ is {\it recoverable by $S_0$} if any $S_0$-universal $\z$-lattice represents $\ell$ itself.   
\end{defn}

Note that the ternary $\z$-lattice $\langle 1,1,2 \rangle$ is recoverable by $T_{1}$. 
We simply say $\ell$ is {\it recoverable} if there is a finite set of proper sublattices satisfying the above property. Note that if $\ell$ is recoverable, then there is a minimal $S$-universality criterion set whose cardinality is greater than $1$, where $S$ is the set of all sublattices of $\ell$.

\begin{lem} \label{not-recover} 
A $\z$-lattice $\ell$ is not recoverable  if and only if there is a $\z$-lattice that represents all proper sublattices of $\ell$, but not $\ell$ itself. 
\end{lem}
\begin{proof} 
Suppose that $\ell$ is not recoverable. Let $S$ be the set of all proper sublattices of $\ell$ and let $S_{0}=\{\ell_1,\ell_2,\dots,\ell_t\}$ be a minimal $S$-universality criterion set. Since we are assuming that $\ell$ is not recoverable, there is a $\z$-lattice $L$ that represents all $\z$-lattices in $S_0$, whereas $L$ does not represent $\ell$ itself.  Now, since the set $S_0$ is an $S$-universality criterion set, $L$ represents all proper sublattices of $\ell$, but not $\ell$ itself.  The converse is trivial.
\end{proof}

\begin{lem}  \label{scaling} 
Let $\ell$ be a $\z$-lattice and let $a$ be a positive integer. If $\ell^a$ is recoverable, then so is $\ell$. 
\end{lem}
\begin{proof}
Assume that $\ell^a$ is recoverable by $\{\ell^a_1,\ell^a_2,\dots,\ell^a_t\}$, where $\ell_i$ is a proper sublattice of $\ell$ for any $i=1,2,\dots,t$.  Let $M$ be any $\z$-lattice that represents $\ell_i$ for any $i$. Then $\ell_i^a \ra M^a$ for any $i$, and hence $\ell^a \ra M^a$. Therefore $\ell \ra M$ and $\ell$ is recoverable by $\{\ell_1,\ell_2,\dots,\ell_t\}$ .  
\end{proof}

\begin{rmk}{\rm
Any unary $\z$-lattice $\ell$ cannot be recoverable. Let $\ell=\langle 1 \rangle$. Note that $\langle 2,2,5 \rangle$ represents all squares of integers except for $1$ (see \cite{jko}). Hence $\langle 2,2,5 \rangle$ represents all proper sublattices of $\ell$, but not $\ell$ itself. Therefore $\ell$ is not recoverable by Lemma \ref{not-recover}. Moreover, since every unary $\z$-lattice can be obtained from $\ell$ by a suitable scaling, it is not recoverable by Lemma \ref{scaling}.}
\end{rmk}

\begin{rmk} {\rm
Note that the converse of the above lemma does not hold in general. Let $\ell=\langle1,4\rangle$ be the binary $\z$-lattice. Let $L$ be any $\z$-lattice representing both $\ell_1=\langle1,16\rangle$ and $\ell_2=\langle4,4\rangle$. Since $L$ represents $\ell_1$,   there is a vector $e_1 \in L$ and a $\z$-sublattice $L_1$ of $L$ such that $L=\z e_1+L_1$, where $Q(e_1)=1$ and $B(e_1,L_1)=0$. Furthermore, since $L$ represents $\ell_2 = \langle4,4\rangle$, there are nonnegative integers $a,b$ and vectors $x,y \in L_1$ such that 
$$
Q(ae_1+x)=a^{2}+Q(x)=Q(be_1+y)=4 \quad \text{and} \quad B(ae_1+x,be_1+y)=0.
$$
If $a=2$, then $x=0$ and $b=0$. Hence $\langle4\rangle \ra L_1$.
If $a=1$, then 
$$
b=0 \quad \text{and}\quad Q(y)=4  \quad\text{or}\quad
b=1, \ Q(x)=Q(y)=3, \ \text{and}\ B(x,y)=-1.
$$
For the latter case, we have $Q(x+y)=4$. Finally, if $a=0$, then $Q(x)=4$. Therefore $L_1$ represents $4$ in any case, which implies that $L$ represents $\ell$. Hence
$\ell$ is recoverable by $\{\ell_1,\ell_2\}$. 

Now, we show that $\ell^2=\langle2,8\rangle$ is not recoverable. To show this, let $S$ be the set of all binary $\z$-lattices with minimum greater than or equal to $9$, and let $S_0=\{\mathfrak m_1,\dots,\mathfrak m_t\}$ be a finite minimal $S$-universality criterion set. Then $\mathfrak m_1\perp \cdots\perp \mathfrak m_t$ represents all binary $\z$-lattices with minimum greater than or equal to $9$.  Now, we define
$$
L=K \perp  \mathfrak m_1\perp \cdots\perp \mathfrak m_t,
\quad
\text{where}
\quad
K=\begin{pmatrix} 2&1&1&0\\1&8&0&0\\1&0&8&4\\0&0&4&10\end{pmatrix}.
$$
Clearly, $\ell^2=\langle2,8\rangle$ is not represented by $L$. We show that any proper sublattice of $\ell^2$ is represented by $L$. 
Let $\ell_{3}$ be any proper sublattice of $\ell^2$. If $\min(\ell_3) \ge 9$, then $\ell_3$ is represented by $\mathfrak m_1\perp\cdots\perp \mathfrak m_t$. Hence we may assume that $\min(\ell_3)=2$ or $8$. For the former case, we have $\ell_3 \simeq \langle2,8m^2\rangle$ for some integer $m\ge2$. Since $\langle 8m^2\rangle \ra \mathfrak m_1\perp \cdots\perp \mathfrak m_t$,  $L$ represents $\ell_3$. For the latter case, one may easily check that $\ell_3$ is isometric to one of the binary $\z$-lattices
$$
\langle 8,2m^2\rangle \quad  \text{and} \quad \begin{pmatrix} 8&4\\4&2+8n^2\end{pmatrix},
$$     
where $m\ge2$ and $n\ge1$. Note that $K$ represents the binary $\z$-lattice $\ell_3$ for $m=2$ or  $n=1$. If $m\ge3$ or $n\ge2$, then one may easily show that 
$$
\langle 8,2m^2\rangle \ra \langle 8,8, 2m^2-8\rangle \ra L \ \ \text{or} \ \ \begin{pmatrix} 8&4\\4&2+8n^2\end{pmatrix} \ra \begin{pmatrix}8&4\\4&10\end{pmatrix} \perp \langle 8n^2-8\rangle \ra L. 
$$
 Therefore $L$ represents all proper sublattices of $\ell^2$, but not $\ell^2$ itself. Consequently, $\ell^2$ is not recoverable.}  
\end{rmk}

One may easily check that every additively indecomposable $\z$-lattice is not recoverable.
We further prove that every indecomposable $\z$-lattice $L$ with rank less than $4$ is not recoverable.

\begin{prop}\label{bin}
Any indecomposable binary $\z$-lattice is not recoverable.
\end{prop}
\begin{proof}
Let $\ell$ be an indecomposable binary $\z$-lattice. Let $\{x,y\}$ be a Minkowski-reduced (ordered) basis for $\ell$, that is, $0 \le 2\vert B(x,y)\vert \le Q(x) \le Q(y)$. Let $S$ be the set of all proper sublattices of $\ell$, and let $S_{0}=\{\ell_{1},\ell_{2},\dots,\ell_{t}\}$ be a minimal $S$-universality criterion set.
If we define $L=\ell_1\perp\dots\perp \ell_t$, then $L$ represents all proper sublattices of $\ell$.  Hence it suffices to show that $L$ does not represent $\ell$ itself. To do this, let $\ell_i=\z x_i+\z y_i$, where $\{x_i,y_i\}$ is a Minkowski reduced basis for $\ell_i$ for any $i=1,2,\dots,t$.  Suppose on the contrary that there is a representation $\phi:\ell \rightarrow L$. Since $\ell_{i}$ is a sublattice of $\ell$ for any $i=1,2,\dots,t$, we may assume, without loss of generality, that $\phi(x)=x_{1}$.  Suppose that $\phi(y) = \alpha x_{1} + \beta y_{1} + z$, where $\alpha,\beta \in \z$ and $z \in  (\z x_{2}+\z y_{2}) \perp \cdots \perp (\z x_{t}+\z y_{t})$.
Since $\ell$ is indecomposable, $\beta$ cannot be zero.
Then, we have
$$
d\ell=d(\phi(\ell)) = d(\z x_{1}+\z (\alpha x_{1} + \beta y_{1} + z)) \ge d(\z x_{1}+\z (\alpha x_{1} + \beta y_{1})) \ge d\ell_{1} > d\ell,
$$
which is a contradiction. Therefore $L$ does not represent $\ell$ and hence, by lemma \ref{not-recover}, the binary $\z$-lattice $\ell$ is not recoverable.
\end{proof}

\begin{prop}
Any  indecomposable ternary $\z$-lattice is not recoverable.
\end{prop}

\begin{proof}
Suppose that there is a ternary $\z$-lattice, say $L$, that is recoverable. Then there are proper sublattices $L_{1}, L_{2}, \dots, L_{t}$ of $L$ such that $L$ is represented by $L_{1} \perp L_{2} \perp \cdots \perp L_{t}$. 
Without loss of generality,  we may assume that all $L_{i}$'s are of rank 3. Let 
$$
\phi : L \ra L_{1} \perp L_{2} \perp \cdots \perp L_{t}
$$
be a representation. Let $\{ u,v,w \}$ be a Minkowski reduced (ordered) basis for $L$. Without loss of generality, we may assume that $\phi(u)=x_{1} \in L_1$. Clearly, there exists a Minkowski reduced basis for $L_{1}$, say  $\{ x_{1}, x_{2}, x_{3} \}$, containing $x_{1}$.
 Assume that 
$$
\phi(v) = a_{1}x_{1}+x+y,
$$
where $a_1\in \z$, $x \in \z x_{2}+ \z x_{3}$ and $y \in L_{2} \perp \cdots \perp L_{t}$.

First, assume that $x=0$.
Since 
$$
2|a_{1}| Q(x_{1}) =2|B(x_{1}, a_{1}x_{1}+y)|=2\vert B(u,v)\vert\le Q(u)=Q(x_{1}),
$$
we have $a_{1}=0$.
Put 
$$
\phi(w) = b_{1}x_{1}+b_{2}x_{2}+b_{3}x_{3}+z,
$$
where $b_i\in \z$ for any $i=1,2,3$ and $z \in L_{2} \perp \cdots \perp L_{t}$.
Then, we have
$$
\phi(L)= \z x_1 +\z y +\z (b_{1}x_{1}+b_{2}x_{2}+b_{3}x_{3}+z).
$$
If $b_{3} \ne 0$, then
$$
\begin{array}{lll}
\mu_{3}(L) = Q(b_{1}x_{1}+b_{2}x_{2}+b_{3}x_{3}+z)\!\!\!&= Q(b_{1}x_{1}+b_{2}x_{2}+b_{3}x_{3})+Q(z)\\[0.2cm]
                                                  \!\!\!&\ge \mu_{3}(L_{1}) +Q(z) \ge \mu_{3}(L)+Q(z),
\end{array}
$$
which implies that $z=0$.
Hence, $\phi(L)$ is decomposable, which is a contradiction.
Therefore, we have $b_{3}=0$.
Observe that 
$$
L_{1} = \z x_{1} + \z x_{2} + \z x_{3} \subseteq L=\z u +\z v+\z w \simeq \phi(L)
$$
Then, $b_{1}x_{1} + b_{2}x_{2} = \alpha u + \beta v + \gamma w$ for some integers $\alpha, \beta$ and $\gamma$.
If $\gamma \ne 0$, then 
$$
\mu_{3}(L) = Q(b_{1}x_{1}+b_{2}x_{2})+Q(z) \ge Q(w)+Q(z) = \mu_{3}(L)+Q(z).
$$
This implies that $z=0$, which is a contradiction.
Hence, $b_{1}x_{1} + b_{2}x_{2} = \alpha u + \beta v$.
Similarly, we have $x_{1} = \alpha_{1} u + \beta_{1} v$ for some integers $\alpha_1$ and $\beta_1$. 
Since $Q(x_{1})=Q(u)$ and $B(u,v)=0$, we have $x_{1} = \pm u$ and $b_{1}x_{1} + b_{2}x_{2} = \beta v$  or $x_{1} = \pm v$ and $b_{1}x_{1} + b_{2}x_{2} = \alpha u$.
In any case, $\phi(L)$ is decomposable, which is a contradiction.

Finally, assume that $x \ne 0$.
Since
$$
\begin{array}{lll}
\mu_{2}(L) = Q(v)=Q(a_{1}x_{1}+x+y) \!\!\!&= Q(a_{1}x_{1}+x)+Q(y)\\[0.2cm]
                                    \!\!\!& \ge  \mu_{2}(L_{1}) +Q(y) \ge \mu_{2}(L) +Q(y),
\end{array}
$$
we have $y=0$.
Put 
$$
\phi(v) = a_{1}x_{1}+a_{2}x_{2}+a_{3}x_{3} \quad \mbox{and}\quad \phi(w) = b_{1}x_{1}+b_{2}x_{2}+b_{3}x_{3}+z
$$
where $a_i,b_i\in \z$ for any $i=1,2,3$ and $z \in L_{2} \perp \cdots \perp L_{t}$.
If $b_{3} \ne 0$, then 
$$
\mu_{3}(L) = Q(b_{1}x_{1}+b_{2}x_{2}+b_{3}x_{3})+Q(z) \ge \mu_{3}(L_{1})+Q(z) \ge \mu_{3}(L)+Q(z).
$$
This implies that $z=0$. Then $\phi(L) \subseteq L_{1}$, which is a contradiction.
Hence, $b_{3}=0$.
Suppose that $a_{3} \ne 0$.
Since 
$$
\mu_{2}(L)=Q(a_{1}x_{1}+a_{2}x_{2}+a_{3}x_{3}) \ge \mu_{3}(L_{1}) \ge \mu_{3}(L),$$
we have $\mu_{2}(L)= \mu_{3} (L) = \mu_{2}(L_{1}) = \mu_{3}(L_{1})$. Then, we have
$$
\mu_{2}(L) = \mu_{3} (L) = Q(b_{1}x_{1}+b_{2}x_{2}+z) \ge \mu_{2}(L_{1})+Q(z) = \mu_{2}(L)+Q(z).
$$
This implies that $z=0$, which is a contradiction.
Therefore $a_{3} = 0$. Since $a_{2} \ne 0$, we have 
$$
\mu_{2}(L) = Q(a_{1}x_{1}+a_{2}x_{2}) \ge \mu_{2}(L_{1}) =Q(x_{2}) \ge \mu_{2}(L).
$$
This means that $Q(a_{1}x_{1}+a_{2}x_{2})=Q(x_{2})$.
Let 
$$
\z x_{1}+\z x_{2} = \begin{pmatrix} s & r \\ r & t \end{pmatrix}.
$$
Then, we have
$$
t=a_{1}^{2}s+2a_{1}a_{2}r+a_{2}^{2}t=s \left( a_{1} + \frac{ra_{2}}{s} \right)^{2} +a_{2}^{2} \left( t-\frac{r^{2}}{s} \right)  \ge a_{2}^{2} \left( t-\frac{s}{4} \right).
$$
If $|a_{2}| \ge 2$, then $t \ge 4t-s>t$, which is a contradiction.
Hence, we have $a_{2} = \pm 1$. Then $\phi(L) = \z x_{1}+\z x_{2} +\z z$ is decomposable, which is a contradiction. This completes the proof.
\end{proof}

\section{Recoverable binary $\mathbb{Z}$-lattices}

In this section, we focus on recoverable binary $\z$-lattices.
We find some necessary conditions and some sufficient conditions for binary $\z$-lattices to be recoverable.

Let $n$ be a positive integer and let $S$ be the set of all binary $\z$-lattices with minimum greater than or equal to $n$. Then there is a finite minimal $S$-universality criterion set $S_{n}=\{m_{1},\dots,m_{t}\}$ by \cite{kko}.  If we define $M=m_{1}\perp \cdots\perp m_{t}$, then $M$ represents all binary $\z$-lattices with minimum greater than or equal to $n$.
In this section, $\MM{n}$ stands for a $\z$-lattice representing all binary $\z$-lattices with minimum greater than or equal to $n$ and $\min(\MM{n})=n$.
From the above argument, such a $\z$-lattice always exists.

\begin{prop}\label{2leab}
Let  $a$ and $b$ be positive integers such that $2 \le a < b$ and $a$ does not divide $b$.  Then the diagonal binary $\z$-lattice $\ell=\langle a, b \rangle $ is not recoverable.
\end{prop}

\begin{proof} It suffices to show that there is a $\z$-lattice such that it represents all proper sublattices of $\ell$, whereas it does not represent $\ell$ itself.  
Let $h$ be a  positive integer such that $h^{2}a < b < (h+1)^{2}a$.  For any $h \ge 2$, we define a $\z$-lattice 
$$
K(h) =\perp_{\substack{2\le i\le h \\ 1 \le j \le [\frac{i}{2}]}}  
\begin{pmatrix}i^{2}a&ija\\ija&j^{2}a+b\end{pmatrix}.
$$
 Now, we define
$$
L(h) = 
\begin{cases} 
\left( \z x + \z y \right) \perp  \MM{b+1}            & \text{if} \ h=1,\\
\left( \z x + \z y \right) \perp K(h) \perp  \MM{b+1} & \text{otherwise},
\end{cases} 
$$
where $\z x + \z y  = \begin{pmatrix} a&1\\1&b \end{pmatrix}$.
We claim that $L(h)$ represents all proper sublattices of $\ell$, whereas it does not represent $\ell$ itself.

First, we will prove that $L(h)$ represents all proper sublattices of $\ell$.
Let $\ell'$ be a proper sublattice of $\ell$. If $\min(\ell') >b$, then $\MM{b+1}$ represents $\ell'$ and so does $L(h)$.  If $\min(\ell') =b$, then $\ell' \simeq \langle b, \alpha^{2}a \rangle$ for some integer $\alpha$ with $\alpha^{2}a  > b$.
Since $\alpha^{2}a$ is represented by $\MM{b+1}$, $L(h)$ represents $\ell'$. Now, assume that $a<\min(\ell')<b$.
Since $4a \le \min(\ell')$, we have $h \ge 2$. Furthermore, there are integers  $i,j$, and $\beta$ with $2 \le i \le h$, $0 \le j \le \left[\frac{i}{2}\right]$ and $\beta \ge 1$ such that
$$
\ell' \simeq \begin{pmatrix} i^{2}a & ija \\ ija & j^{2}a+\beta^{2}b \end{pmatrix}.
$$
If $\beta=1$, then clearly $\ell'  \ra L(h)$.
Assume that $\beta \ge 2$.
Since $(\beta^{2}-1)b>b$, we have
$$
\ell' \simeq \begin{pmatrix} i^{2}a & ija \\ ija & j^{2}a+\beta^{2}b \end{pmatrix}
\ra \begin{pmatrix} i^{2}a & ija \\ ija & j^{2}a+b \end{pmatrix} \perp \MM{b+1},
$$
which implies that $L(h)$ represents $\ell'$. Finally, if $\min(\ell') = a$, then $\ell' \simeq \langle a, \beta^{2}b \rangle$ for some integer $\beta  \ge 2$. Since $\beta^{2}b$ is represented by $\MM{b+1}$, $L(h)$ represents $\ell'$.

Next, we will show that $L(h)$ does not represent $\ell$.  Clearly, $L(1)$ does not represent $\ell$.  Assume $h \ge 2$. For any $i,j$ with $2 \le i \le h$ and $ 1 \le j  \le \left[ \frac{i}{2} \right]$, let
$$
K_{ij}=\z z + \z w =\begin{pmatrix} i^{2}a&ija\\ija&j^{2}a+b \end{pmatrix}.
$$
 Since
$$
Q(sz+tw) = (si+tj)^{2}a + t^{2}b >a,
$$
for any integers $s$ and $t$, the binary $\z$-lattice $K_{ij}$ does not represent $a$. Suppose that $Q(sz+tw)=b$ for some integers $s$ and $t$. Since $a$ does not divide $b$,  we have $t^{2}=1$ and $si+tj=0$. Furthermore, since $j = |si| \le \left[ \frac{i}{2} \right]$,
we have $s=j=0$. This is a contradiction. Hence $K_{ij}$ does not represent $b$ for any possible integers $i$ and $j$.
Therefore we have
\begin{equation}\label{qvalue}
Q(K(h)) \subseteq (\left\{ ma+nb  \mid m,n \in \n \cup\{0\} \right\} \setminus \left\{ a, b \right\}).
\end{equation}
Suppose that $L(h)$ represents $\ell$.
Let $u \in L(h)$ be a vector with $Q(u) = b$. Then 
$$
u=\alpha x  + \beta y + z + w,
$$ 
for some integers $\alpha, \beta$ and some vectors $z \in K(h)$ and $w \in \MM{b+1}$.
Since 
$$
Q(u) = Q(\alpha x + \beta y) + Q(z) + Q(w)=b \quad \text{and}\quad Q(w) > b,
$$
we have $w=0$. By \eqref{qvalue}, we have $Q(z)=0$ or $Q(z) = \delta a$ for some integer $\delta \ge 2$.
If $|\beta| \ge 2$, then $Q(\alpha x  + \beta y) \ge \beta^{2}(b-1) >b $.
If $\beta=0$, then $Q(\alpha x)$ is a multiple of $a$, and so is $Q(u)$.
Hence, we have $|\beta|=1$. This implies that
$$
u=
\begin{cases} 
 \pm (x-y) \ \text{or} \ \pm y    & \text{if} \ a=2,\\
 \pm y& \text{if} \ a \ge 3.
\end{cases}
$$
Let $v \in L(h)$ be a vector with $Q(v) = a$. Then $v=x$ or $v=-x$. Finally, we have $B(u, v) \ne 0$, which is a contradiction. This completes the proof.  
\end{proof}

\begin{prop}\label{odd}
For any positive odd integer $m$, the diagonal binary $\z$-lattice $\ell=\langle 1, m \rangle$ is not recoverable.
\end{prop}

\begin{proof}
For a positive odd integer $m$, let $\ell=\langle 1, m \rangle$ be the diagonal $\z$-lattice.
Since $\langle 1 \rangle \perp \MM{2}$ represents all proper sublattices of $\langle 1,1 \rangle$, but not $\langle 1,1 \rangle$ itself, the binary $\z$-lattice $\langle 1, 1 \rangle$ is not recoverable.
Hence we may assume that $m \ge 3$.
Let $N$ be any quinary $\z$-lattice that represents all binary even $\z$-lattices.
Note that such a $\z$-lattice exists, for example, the root lattice $D_{5}$ is one of such quinary $\z$-lattices (see \cite{jko2}). Define a $\z$-lattice
$$
L=\langle 1 \rangle \perp N \perp \MM{m+1}.
$$
It is obvious that $\langle 1, m \rangle$ is not represented by $L$.

Let $\ell_1$ be any proper $\z$-sublattice of $\ell$.
First, suppose that $\min(\ell_1)=1$. Then $\ell_1 \simeq \langle 1, m\beta^{2} \rangle$ for some integer $\beta \ge 2$.
Since $\langle m\beta^{2} \rangle  \ra \MM{m+1}$, we have $\ell  \ra  L$.

Now, suppose that $\min(\ell_1)>1$. Then clearly, $\min(\ell_1) \ge 3$.
Choose a Minkowski reduced basis for $\ell_1$ so that for some integers $a,b$, and $c$ with $0 \le 2b \le a \le c$ such that
$$
\ell_1 \simeq \begin{pmatrix} a&b\\ b&c \end{pmatrix}. 
$$
Note that we are assuming that $a\ge 3$.  If $a \equiv c \equiv 0 \pmod{2}$, then $\ell_1\ra N$ and so $\ell_1 \ra L$.  If $a \equiv c \equiv 1 \pmod{2}$, then we define a $\z$-lattice
$$
\ell_1' = \begin{pmatrix} a-1&b-1\\ b-1&c-1 \end{pmatrix}.
$$
Since $d\ell_1' \ge \frac{3}{4}ac-c =\frac{3c}{4}(a-\frac43) > 0$, the even $\z$-lattice $\ell_1'$ is positive definite.
Hence $\ell_1' \ra N$ and therefore $\ell_1  \ra \langle 1 \rangle \perp N \ra  L$.  If $a \equiv 1 \pmod{2}$ and $c\equiv 0 \pmod{2}$, then we define a $\z$-lattice
$$
\ell_1'' = \begin{pmatrix} a-1&b\\ b&c \end{pmatrix}.
$$
Since $d\ell_1'' = \left( \frac{ac}{4}-b^{2} \right) +\frac{c}{4} \left( 3a-4 \right) >0$, the even $\z$-lattice $\ell_1''$ is positive definite.
Hence $\ell_1''  \ra  N$ and therefore $\ell_1  \ra  \langle 1\rangle \perp \ell_1'' \ra \langle 1\rangle \perp N \ra L$.  Since the proof of the case when $a \equiv 0 \pmod{2}$ and $c\equiv 1 \pmod{2}$ is quite similar to this, the proof is left to the readers. 
\end{proof}

\begin{prop} \label{2m4}
For any positive integer $m$ with $m \equiv 2 \pmod{4}$, the diagonal binary $\z$-lattice $\ell=\langle 1, m \rangle$ is not recoverable.
\end{prop}

\begin{proof}
For a positive integer $m \equiv 2 \pmod{4}$, let $\ell=\langle 1,m \rangle$ be the diagonal $\z$-lattice. Since $\langle 1,1 \rangle \perp \MM{3}$ represents all proper sublattices of $\langle 1,2 \rangle$, but it does not represent $\langle 1,2 \rangle$ itself, the binary $\z$-lattice $\langle 1,2 \rangle$ is not recoverable.  If we define
$$
L=\langle 1 \rangle \perp \begin{pmatrix} 4 & 0 & 2 \\ 0 & 5 & 1 \\ 2 & 1 & 7 \end{pmatrix} \perp \MM{7},
$$
then one may check that $L$ represents all proper sublattices of $\langle 1, 6 \rangle$, but it does not represent $\langle 1, 6 \rangle$ itself.
From now on, we assume that $m \ge 10$.
Define
$$
L_m' = \z x + \z y + \z z + \z t = \langle 1,3,5,m-1 \rangle.
$$
Let $N$ be an even 2-universal quinary $\z$-lattice and let $\mathcal{N}$ be the $\z$-lattice obtained from $N$ by scaling the quadratic space $\mathbb{Q} \otimes N$ by $2$. Hence $\mathcal N$ represents all binary $\z$-lattices whose norm is contained in $4\z$. 
Now, we define 
$$
L_m = L_m' \perp \mathcal{N} \perp \MM{m+1}.
$$
We will show that any proper sublattice of $\ell=\langle 1, m \rangle$ is represented by $L_m$,
whereas $\ell$ itself is not represented by $L_m$. Suppose, on the contrary, that $\langle 1, m \rangle$ is represented by $L_m$.
Then, one may easily check that 
$$
\langle m \rangle \longrightarrow \langle 3,5 \rangle \perp \mathcal{N}.
$$
Hence we have $m \equiv 3a^2+5b^2 \pmod{4}$ for some integers $a$ and $b$, which is a contradiction to the fact that $m \equiv 2 \pmod{4}$.

Let $\ell= \z u +\z v = \langle 1, m \rangle$ and let $\ell_1$ be any proper sublattice of $\ell$.  If $\min(\ell_1)=1$, then $\ell_1\simeq \langle1,m\beta^2 \rangle$ for some integer $\beta\ge2$. Clearly, $\ell_1 \ra L_m$.  Assume that $1< \min(\ell_1)< m$. Then there are integers $a,b$, and $c$ such that $\ell_1= \z (au) + \z (bu+ cv)$. If $|c| \ge 2$, then we have $\ell_1 \subseteq \z u  + \z (cv) = \langle 1, c^{2}m \rangle \ra L_m$.
Hence we may assume that $\ell_1 = \z (au)+\z (bu+v)$, where the integers $a$ and $b$ satisfy $a \ge 2$ and $0 \le b < a$. Note that
$$
\ell_1 = \begin{pmatrix} a^{2} & ab \\ ab & b^{2}+m \end{pmatrix}.
$$
First, assume that $a \equiv b \equiv 0 \pmod{2}$.
Since 
$$
\begin{pmatrix} a^{2} & ab \\ ab & b^{2}+m-6 \end{pmatrix} \ra \mathcal{N},
$$
we have $\ell_1\ra \langle 1,5 \rangle \perp \mathcal{N} \ra L_m$.  Now, assume that $a \equiv 0 \pmod{2}$ and $b\equiv 1 \pmod{2}$.
Since 
$$
\begin{pmatrix} a^{2} & ab \\ ab & b^{2}+m-3 \end{pmatrix} \ra \mathcal{N},
$$
we have $\ell_1 \ra \langle 3 \rangle \perp \mathcal{N} \ra L_m$. Assume that $a \equiv b \equiv 1 \Mod{2}$. Let $w\in\mathcal{N}$ be a vector with $Q(w)=m-2$. Then, we have
$$
\z (x+y+z)+\z(y+w)=\begin{pmatrix} 9&3\\ 3&1+m \end{pmatrix} \ra L_m.
$$
Hence we may assume that $a \ge 5$. Consider the following $\z$-lattice 
$$
\ell_1' = \begin{pmatrix} a^{2}-9 & ab-3 \\ ab-3 & b^{2}+m-3 \end{pmatrix}.
$$
Since $m > a^2$, we have $d (\ell_1')>0$. Hence $\ell_1' \ra \mathcal{N}$.
Therefore there are vectors $w_{1}, w_{2} \in \mathcal{N}$ such that
$$
\ell_1' \simeq \z w_{1}+\z w_{2} \subseteq \mathcal{N}.
$$
Then, we have
$$
\z (x+y+z+w_{1}) + \z (y+w_{2})  = \begin{pmatrix} a^{2} & ab \\ ab & b^{2}+m \end{pmatrix}.
$$
This implies that $\ell_1$ is represented by $L_m$. 

Finally, assume that $a \equiv 1 \pmod{2}$ and $b\equiv 0 \pmod{2}$.
If $a=3$, then $b=0$ or $2$. If $b=0$, then 
$$
\ell_1 = \langle 9,m \rangle \longrightarrow \langle 1,5,m-1 \rangle \perp \mathcal{N} \ra L_m.
$$
If $b=2$, then we have 
$$
\ell_1 = \begin{pmatrix} 9 & 6 \\ 6 & m+4 \end{pmatrix} \simeq \begin{pmatrix} 9 & 3 \\ 3 & 1+m \end{pmatrix} \ra L_m.
$$
Now, assume that $a \ge 5$. Consider the following $\z$-lattice 
$$
\ell_1'' = \begin{pmatrix} a^{2}-9 & ab-4 \\ ab-4 & b^{2}+m-6 \end{pmatrix}.
$$
Since $d (\ell_1'')>0$, we have $\ell_1'' \ra \mathcal{N}$.
Hence there are vectors $w'_{1}, w'_{2} \in \mathcal{N}$ such that
$$
\ell_1'' \simeq \z w'_{1}+\z w'_{2} \subseteq \mathcal{N}.
$$
Then, we have
$$
\z (x+y+z+w'_{1}) + \z (-x+z+w'_{2})  = \begin{pmatrix} a^{2} & ab \\ ab & b^{2}+m \end{pmatrix}.
$$
Therefore, we have $\ell_1 \ra L_m$. 

 If $\min(\ell_1)=m$, then $\ell_1 \simeq \langle \alpha^2,m \rangle$ for some integer $\alpha$ with $\alpha^2>m$. Hence, we have $\ell_1 \ra \langle 1,m-1\rangle\perp \MM{m+1} \ra L_m$.  Finally, if $\min(\ell_1)>m$, then we have $\ell_1 \ra  \MM{m+1} \ra L_m$. This completes the proof. 
\end{proof}

\section{Recoverable numbers}

From Propositions \ref{bin}, \ref{2leab}, \ref{odd}, and \ref{2m4},  one may conclude that if a binary $\z$-lattice $\ell$ is recoverable, then $\ell = \langle a, 4ma \rangle$ for some positive integers $a$ and  $m$.
In this section, we focus on the case when $a=1$.  A positive integer $m$ is called \textit{recoverable} if the diagonal binary $\z$-lattice $\langle 1,4m \rangle$ is recoverable. We prove that any square of an integer is a recoverable number. We also prove that there are infinitely many non square recoverable numbers.

\begin{prop}
Any square of an integer is recoverable, that is,  the diagonal binary $\z$-lattice $\ell = \langle 1, 4m^{2} \rangle$ is recoverable for any integer $m$.  
\end{prop}
\begin{proof}
Let $\ell= \langle 1, 4m^{2} \rangle$ be the diagonal binary $\z$-lattice. 
Let $S$ be the set of all proper sublattices of $\ell$. By Lemma \ref{not-recover}, it suffices to show that any $S$-universal $\z$-lattice represents $\ell$ itself.  Let $L$ be an $S$-universal $\z$-lattice. Since $\langle 1, 16m^{2} \rangle  \ra  L$, we have $L=\z e_{1} \perp L' =\langle 1 \rangle \perp L'$ for some $\z$-lattice $L'$.
Since $\langle 4, 4m^{2} \rangle \ra L$, one of the following holds:
\begin{enumerate}
\item  there is a vector $y \in L'$ such that $\z (2e_{1}) + \z y = \langle 4, 4m^{2} \rangle$;
\item  there are vectors $x,y \in L'$ and an integer $a$ such that 
$$
\z (e_{1}+x) + \z (ae_{1}+y) = \langle 4, 4m^{2} \rangle;
$$
\item  there are vectors $x,y \in L'$ and an integer $a$ such that 
$$
\z x + \z (ae_{1}+y) = \langle 4, 4m^{2} \rangle.
$$
\end{enumerate}
If (1) holds, then $Q(y)=4m^{2}$. Hence $L$ represents $\ell$.   If (2) holds, then 
$$
\z x+\z y = \begin{pmatrix} 3&-a\\ -a & 4m^{2}-a^{2}\end{pmatrix}.
$$
Hence we have $Q(ax+y)=4m^{2}$. Therefore $L$ represents $\ell$.   Finally, if (3) holds, then we have $Q(mx)=4m^{2}$. Therefore $L$ represents $\ell$.  This completes the proof.
\end{proof}

Let $\mathscr{L}$ be the set of all isometry classes of binary $\z$-lattices and let $\mathscr{L}_{13}$ be the set of all isometry classes of binary $\z$-lattices whose  second successive minimum is greater than or equal to $13$.
We define a map $\phi_{9} : \mathscr{L}_{13}  \ra  \mathscr{L}$ by
$$
\phi_{9} \left( \begin{pmatrix} a & b \\ b & c \end{pmatrix}  \right) = \begin{pmatrix} a & b \\ b & c-9 \end{pmatrix},
$$
where $\begin{pmatrix} a & b \\ b & c \end{pmatrix}$ is a Minkowski-reduced form in the class so that $0 \le 2b \le a \le c$. 
Since 
$$
d(\phi_{9}(K))=ac-b^2-9a=\left(\frac {ac}4-b^2\right)+\frac{3a}4(c-12)>0,
$$
the above map $\phi_9$ is well-defined.

\begin{lem}\label{phi9}
Let $L$ be a $\z$-lattice and let $K$ be a binary $\z$-lattice in $\mathscr{L}_{13}$.
If $\phi_{9}^{k}(K)$ is represented by $L$ for some nonnegative integer $k$,
then 
$$
K \longrightarrow L \perp 9I_{5}.
$$
Here, $9I_{5}$ is the quinary $\z$-lattice obtained from the cubic lattice  $I_{5}$ by scaling the quadratic space $\mathbb{Q} \otimes I_{5}$ by $9$. 
\end{lem}

\begin{proof}
Let $L$ be a $\z$-lattice and let $K$ be a binary $\z$-lattice in $\mathscr{L}_{13}$. Let $\begin{pmatrix} a & b \\ b & c \end{pmatrix}$ be the Minkowski-reduced form in the isometry class of $K$.
Note that $9I_{5}$ represents all binary $\z$-lattices whose scale is contained in $9\z$. We use an induction on $k$. Suppose that $\phi_{9}(K)$ is represented by $L$.
Then, it is obvious that 
$$
K= \begin{pmatrix} a & b \\ b & c \end{pmatrix} = \begin{pmatrix} a & b \\ b & c-9 \end{pmatrix} + \begin{pmatrix} 0 & 0 \\ 0 & 9 \end{pmatrix} \ra  L \perp 9I_{5}.
$$
Suppose that the assertion is true for $k$. Assume that $\phi_{9}^{k+1}(K) \ra L$.
Let $K'=\phi_{9}(K)$.
Then $\phi_{9}^{k}(K') = \phi_{9}^{k+1}(K) \ra L$.
It follows from the induction hypothesis that $K' \ra  L \perp 9I_{5}$.
This implies that 
$$
K' =\begin{pmatrix} a & b \\ b & c-9 \end{pmatrix} = \begin{pmatrix}	\alpha_{1} & \beta_{1} \\ \beta_{1} & \gamma_{1} \end{pmatrix} + \begin{pmatrix} \alpha_{2} & \beta_{2} \\ \beta_{2} & \gamma_{2} \end{pmatrix},
$$
where $\begin{pmatrix}	\alpha_{1} & \beta_{1} \\ \beta_{1} & \gamma_{1} \end{pmatrix}  \ra  L$ and $\begin{pmatrix} \alpha_{2} & \beta_{2} \\ \beta_{2} & \gamma_{2} \end{pmatrix} \ra 9I_{5}$.
Since 
$
\begin{pmatrix} \alpha_{2} & \beta_{2} \\ \beta_{2} & \gamma_{2}+9 \end{pmatrix}  \ra  9I_{5},
$
we have
$$
K  = \begin{pmatrix} a & b \\ b & c \end{pmatrix}
=\begin{pmatrix}	\alpha_{1} & \beta_{1} \\ \beta_{1} & \gamma_{1} \end{pmatrix} + \begin{pmatrix} \alpha_{2} & \beta_{2} \\ \beta_{2} & \gamma_{2}+9 \end{pmatrix}
\ra L \perp 9I_{5}.
$$
This completes the proof.
\end{proof}

\begin{lem}\label{ord3m}
Any proper sublattice of $\langle 1,1 \rangle$ is represented by both
$$
L_{1}= \langle 1,2,3 \rangle \perp \begin{pmatrix} 2 & 1 \\ 1 & 5 \end{pmatrix} \perp 9I_{5} \  \mbox{ and }\  L_{2}=\langle 1,2,6 \rangle \perp \begin{pmatrix} 2 & 1 \\ 1 & 5 \end{pmatrix} \perp 9I_{5}.
$$
\end{lem}
\begin{proof}
Since the proof is quite similar to each other, we only provide the proof of the first case.
Let $\ell$ be any proper sublattice of $\langle 1,1 \rangle$. 

 One may directly check that the quinary $\z$-lattice $\langle 1,2,3 \rangle \perp \begin{pmatrix} 2 & 1 \\ 1 & 5 \end{pmatrix}$ represents all binary $\z$-lattices whose second successive minimum is less than or equal to $12$ except for the following $15$ binary $\z$-lattices:
 
\begin{equation}\label{15lattices}
\begin{array}{ll}
&
\begin{pmatrix} 1 & 0 \\ 0 & 1 \end{pmatrix}, \begin{pmatrix} 1 & 0 \\ 0 & 6 \end{pmatrix},
\begin{pmatrix} 2 & 1 \\ 1 & 2 \end{pmatrix}, \begin{pmatrix} 2 & 1 \\ 1 & 3 \end{pmatrix},
\begin{pmatrix} 2 & 1 \\ 1 & 4 \end{pmatrix},\\[0.3cm]

&
\begin{pmatrix} 4 & 0 \\ 0 & 6 \end{pmatrix}, \begin{pmatrix} 4 & 1 \\ 1 & 4 \end{pmatrix}, \begin{pmatrix} 4 & 1 \\ 1 & 13 \end{pmatrix},\begin{pmatrix} 4 & 2 \\ 2 & 7 \end{pmatrix}, \begin{pmatrix} 6 & 0 \\ 0 & 7 \end{pmatrix}, \\[0.3cm]

&
\begin{pmatrix} 6 & 0 \\ 0 & 10 \end{pmatrix}, \begin{pmatrix} 6 & 3 \\ 3 & 7 \end{pmatrix}, \begin{pmatrix} 6 & 3 \\ 3 & 10 \end{pmatrix}, \begin{pmatrix} 7 & 1 \\ 1 & 10 \end{pmatrix}, \begin{pmatrix} 10 & 2 \\ 2 & 10 \end{pmatrix}. 
\end{array}
\end{equation}
Note that the above $15$ binary $\z$-lattices are not proper sublattices of $\langle 1,1 \rangle$. 
Hence we may  assume that $\mu_2(\ell)\ge 13$. Since $\mu_{2}(\phi_{9}(\ell)) \le \max\{\mu_{1}(\ell), \mu_{2}(\ell)-9 \}$, there exists a positive integer $k$ such that $\phi_{9}^{k-1}(\ell) \in \mathscr{L}_{13}$ and $\mu_{2}(\phi_{9}^{k}(\ell)) \le 12$. When $k=1$, then we let $\phi_9^0(\ell)=\ell$. If 
$$
\phi_{9}^{k}(\ell)  \ra  \langle 1,2,3 \rangle \perp \begin{pmatrix} 2 & 1 \\ 1 & 5 \end{pmatrix},
$$ 
then, by Lemma \ref{phi9}, we have
$$
\ell  \ra \langle 1,2,3 \rangle \perp \begin{pmatrix} 2 & 1 \\ 1 & 5 \end{pmatrix} \perp 9I_{5}.
$$
Hence, we may assume that $\phi_{9}^{k}(\ell)$ is isometric to one of 15 binary  $\z$-lattices listed in \eqref{15lattices}.
Since $d\ell$ is a square of an integer and $d(\phi_{9}(\ell)) = d\ell -9\mu_{1}(\ell)$, we see that $\ord_{3}(d(\phi_{9}^{k}(\ell)))$ cannot be one. Hence  $\phi_{9}^{k}(\ell)$ is isometric to one of the following $\z$-lattices:
$$
\begin{pmatrix} 1 & 0 \\ 0 & 1 \end{pmatrix}, \begin{pmatrix} 2 & 1 \\ 1 & 3 \end{pmatrix}, \quad \text{and} \quad \begin{pmatrix} 2 & 1 \\ 1 & 4 \end{pmatrix}.
$$
Since $\phi_{9}^{k-1}(\ell) \in \mathscr{L}_{13}$ and
$$
\begin{array}{ll}
&
\phi_{9}^{-1}\left(\begin{pmatrix} 1 & 0 \\ 0 & 1 \end{pmatrix}\right)= \left\{ \begin{pmatrix} 1 & 0 \\ 0 & 10 \end{pmatrix}, \begin{pmatrix} 2 & 1 \\ 1 & 10 \end{pmatrix}, \begin{pmatrix} 5 & 2 \\ 2 & 10 \end{pmatrix}, \begin{pmatrix} 10 & 3 \\ 3 & 10 \end{pmatrix}\right\},\\[0.3cm]

&
\phi_{9}^{-1}\left(\begin{pmatrix} 2 & 1 \\ 1 & 3 \end{pmatrix}\right)= \left\{ \begin{pmatrix} 2 & 1 \\ 1 & 12 \end{pmatrix}, \begin{pmatrix} 3 & 1 \\ 1 & 11 \end{pmatrix}, \begin{pmatrix} 7 & 3 \\ 3 & 11 \end{pmatrix}, \begin{pmatrix} 10 & 5 \\ 5 & 12 \end{pmatrix} \right\},\\[0.3cm]

&
\phi_{9}^{-1}\left(\begin{pmatrix} 2 & 1 \\ 1 & 4 \end{pmatrix}\right)= \left\{ \begin{pmatrix} 2 & 1 \\ 1 & 13 \end{pmatrix}, \begin{pmatrix} 4 & 1 \\ 1 & 11 \end{pmatrix}, \begin{pmatrix} 8 & 3 \\ 3 & 11 \end{pmatrix}\right\},
\end{array}
$$
we have 
$$
\phi_{9}^{k}(\ell)\simeq \begin{pmatrix} 2 & 1 \\ 1 & 4 \end{pmatrix}\quad \text{and}\quad
\phi_{9}^{k-1}(\ell)\simeq \begin{pmatrix} 2 & 1 \\ 1 & 13 \end{pmatrix}.
$$
One may check that 
$$
\phi_{9}^{k-1}(\ell)\simeq\begin{pmatrix} 2 & 1 \\ 1 & 13 \end{pmatrix} \ra \langle 1,2,3 \rangle \perp \begin{pmatrix} 2 & 1 \\ 1 & 5 \end{pmatrix}.
$$
Hence, by Lemma \ref{phi9}, we have
$$
\ell  \longrightarrow \langle 1,2,3 \rangle \perp \begin{pmatrix} 2 & 1 \\ 1 & 5 \end{pmatrix} \perp 9I_{5}.
$$
This completes the proof.
\end{proof}

\begin{prop}
If $m$ is a positive integer with $\ord_{3}(m)=1$, then $m$ is not a recoverable number.
\end{prop}
\begin{proof}
Let $m$ be a positive integer with $\ord_{3}(m)=1$. Then, we write $m=3m'$ with $m'\equiv 1$ or $2 \Mod 3$. We define
$$
L_m =
\begin{cases}
\langle 1,2,6,4m-1 \rangle \perp \begin{pmatrix} 2 & 1 \\ 1 & 5 \end{pmatrix} \perp 9I_{5}\perp \MM{4m+1} & \text{if $m'\equiv 1 \Mod 3$},\\[0.5cm]

\langle 1,2,3,4m-1 \rangle \perp \begin{pmatrix} 2 & 1 \\ 1 & 5 \end{pmatrix} \perp 9I_{5}\perp \MM{4m+1} & \text{if $m'\equiv 2 \Mod 3$}.
\end{cases}
$$
Clearly, $L_m$ does not represent $\langle 1, 4m \rangle$. As in the proof of Lemma \ref{2m4}, it is enough to show that $L_m$ represents every proper sublattice of $\langle 1,4m \rangle$, which is of the form $\begin{pmatrix} a^{2} & ab \\ ab & b^{2}+4m \end{pmatrix}$ for some integers $a$ and $b$ with $a \ge 2$ and $0\le 2b \le a$. By Lemma \ref{ord3m}, we have 
$$
\begin{pmatrix} a^{2} & ab \\ ab & b^{2}+1 \end{pmatrix}
\ra \langle 1,2,3 \rangle \perp \begin{pmatrix} 2 & 1 \\ 1 & 5 \end{pmatrix} \perp 9I_{5} \
\text{and} \ \langle 1,2,6 \rangle \perp \begin{pmatrix} 2 & 1 \\ 1 & 5 \end{pmatrix} \perp 9I_{5},
$$
which implies that
$$
\begin{pmatrix} a^{2} & ab \\ ab & b^{2}+4m \end{pmatrix} \ra L_m.
$$
This completes the proof.
\end{proof}

\begin{prop}\label{12lattices}
Any integer $m$ is a recoverable number if $4m$ is represented by all of the following binary $\z$-lattices
$$
\begin{array}{lll}
&\begin{pmatrix} 2 & 1 \\ 1 & 4 \end{pmatrix},
\begin{pmatrix} 3 & 1  \\ 1 & 4 \end{pmatrix},
\begin{pmatrix} 4 & 0 \\ 0 & 4 \end{pmatrix},
\begin{pmatrix} 4 & 0 \\ 0 & 5 \end{pmatrix},
\begin{pmatrix} 4 & 1 \\ 1 & 6 \end{pmatrix},
\begin{pmatrix} 4 & 1 \\ 1 & 7 \end{pmatrix},\\[0.3cm]

&\begin{pmatrix} 4 & 0 \\ 0 & 8 \end{pmatrix},
\begin{pmatrix} 4 & 1 \\ 1 & 8 \end{pmatrix},
\begin{pmatrix} 4 & 2 \\ 2 & 8 \end{pmatrix},
\begin{pmatrix} 4 & 0 \\ 0 & 9 \end{pmatrix},
\begin{pmatrix} 4 & 1 \\ 1 & 9 \end{pmatrix}, \quad \text{and}\quad
\begin{pmatrix} 4 & 2 \\ 2 & 9 \end{pmatrix}.
\end{array}
$$
\end{prop}
\begin{proof}
Let $L$ be any $\z$-lattice that represents all proper sublattices of $\langle 1,4m \rangle$.
Since $\langle 1, 16m \rangle  \ra  L$, we have  $L=\z e_1 \perp L'= \langle 1 \rangle \perp L'$ for some $\z$-lattice $L'$. To prove the proposition, it suffices to show that $\langle 1,4m\rangle$ is represented by $L$, that is, $4m$ is represented by $L'$. 
 
Since $\langle 4, 4m \rangle  \ra  L$, one of the followings holds:
\begin{enumerate}
\item there is a vector $y \in L'$ such that $\z (2e_{1}) + \z y = \langle 4, 4m \rangle$;
\item there are vectors $x,y \in L'$ and an integer $a$ such that 
$$
\z (e_{1}+x) + \z (ae_{1}+y) = \langle 4, 4m \rangle;
$$
\item there are vectors $x,y \in L'$ and an integer $a$ such that 
$$
\z x + \z (ae_{1}+y) = \langle 4, 4m \rangle.
$$
\end{enumerate}
If (1) or (2) holds, then $4m$ is represented by $L'$. Therefore we may assume that (3) holds. Hence $4$ is represented by $L'$. 

Now, note that  $\langle 9,4m \rangle$ is also represented by $L$. Similarly to the above, one may easily show that $4m$ is represented by $L'$ or 
one of binary $\z$-lattices $\begin{pmatrix} 8 & -a \\ -a & 4m-a^2 \end{pmatrix}$ and $\langle 9, 4m-a^2\rangle$ is represented by $L'$. Hence  $8$ or $9$ is represented by $L'$.

Suppose that $L'$ represents 4 and 8. Then $L'$ represents at least one of the following binary $\z$-lattices:
$$
\begin{pmatrix} 4 & 0 \\ 0 & 8 \end{pmatrix},
\begin{pmatrix} 4 & 1 \\ 1 & 8 \end{pmatrix},
\begin{pmatrix} 4 & 2 \\ 2 & 8 \end{pmatrix},
\begin{pmatrix} 4 & 3 \\ 3 & 8 \end{pmatrix},
\begin{pmatrix} 4 & 4 \\ 4 & 8 \end{pmatrix}, \quad \text{and}\quad
\begin{pmatrix} 4 & 5 \\ 5 & 8 \end{pmatrix}.
$$
Here, we have 
$$
\begin{pmatrix} 4 & 3 \\ 3 & 8 \end{pmatrix} \simeq \begin{pmatrix} 4 & 1 \\ 1 & 6 \end{pmatrix},
\begin{pmatrix} 4 & 4 \\ 4 & 8 \end{pmatrix} \simeq \begin{pmatrix} 4 & 0 \\ 0 & 4 \end{pmatrix},\quad \text{and}\quad
\begin{pmatrix} 4 & 5 \\ 5 & 8 \end{pmatrix} \simeq \begin{pmatrix} 2 & 1 \\ 1 & 4 \end{pmatrix}.
$$

Finally, suppose that $L'$ represents 4 and 9.
Then $L'$ represents at least one of the following binary $\z$-lattices:
$$
\begin{pmatrix} 4 & 0 \\ 0 & 9 \end{pmatrix},
\begin{pmatrix} 4 & 1 \\ 1 & 9 \end{pmatrix},
\begin{pmatrix} 4 & 2 \\ 2 & 9 \end{pmatrix},
\begin{pmatrix} 4 & 3 \\ 3 & 9 \end{pmatrix},
\begin{pmatrix} 4 & 4 \\ 4 & 9 \end{pmatrix}, \quad\text{and}\quad
\begin{pmatrix} 4 & 5 \\ 5 & 9 \end{pmatrix},
$$
Here, we have 
$$
\begin{pmatrix} 4 & 3 \\ 3 & 9 \end{pmatrix} \simeq \begin{pmatrix} 4 & 1 \\ 1 & 7 \end{pmatrix},
\begin{pmatrix} 4 & 4 \\ 4 & 9 \end{pmatrix} \simeq \begin{pmatrix} 4 & 0 \\ 0 & 5 \end{pmatrix}, \quad\text{and}\quad
\begin{pmatrix} 4 & 5 \\ 5 & 9 \end{pmatrix} \simeq \begin{pmatrix} 3 & 1 \\ 1 & 4 \end{pmatrix}.
$$
Therefore, if $4m$ is represented by all of the above $12$ binary $\z$-lattices, then $4m$ is represented by $L'$. This completes the proof. 
\end{proof}

\begin{cor}
Let $m \equiv 1 \pmod 8$ be a prime. If $m$ is a quadratic residue modulo $q$ for any prime $q\in\{3,5,7,11,23,31\}$, then $m$ is a recoverable number. In particular, any prime $m \equiv 5569 \pmod {3\cdot 5\cdot 7\cdot 11\cdot 23\cdot 31}$ is a recoverable number. Therefore there are infinitely many non square recoverable numbers. 
\end{cor}

\begin{proof}
Note that $4m$ is represented by all of $12$ binary $\z$-lattices in Proposition \ref{12lattices} if and only if $m$ is represented by all of the following binary $\z$-lattices
\begin{equation}\label{9lattices}
\begin{array}{lll}
&\hspace{0.7cm}\begin{pmatrix} 1 & 0 \\ 0 & 1 \end{pmatrix},
\begin{pmatrix} 1 & 0  \\ 0 & 2 \end{pmatrix},
\begin{pmatrix} 1 & \frac12 \\ \frac12 & 2 \end{pmatrix},
\begin{pmatrix} 1 & \frac12 \\ \frac12 & 3 \end{pmatrix},\\[0.3cm]
&
\begin{pmatrix} 1 & 0 \\ 0 & 5 \end{pmatrix},
\begin{pmatrix} 1 & \frac12 \\ \frac12 & 7 \end{pmatrix},
\begin{pmatrix} 1 & 0 \\ 0 & 8 \end{pmatrix},
\begin{pmatrix} 1 & \frac12 \\ \frac12 & 9 \end{pmatrix}, \quad\text{and}\quad
\begin{pmatrix} 1 & 0 \\ 0 & 9 \end{pmatrix},
\end{array}
\end{equation}
and furthermore, $m$ is represented by both 
\begin{equation}\label{2genera}
\gen\left(
\begin{pmatrix} 1 & \frac12 \\ \frac12 & 6 \end{pmatrix}\right) \quad \text{and} \quad \gen\left(\begin{pmatrix} 1 & \frac12 \\ \frac12 & 8 \end{pmatrix}\right).
\end{equation}
As a sample, assume that $4m$ is represented by $\begin{pmatrix} 4 & 1 \\ 1 & 6 \end{pmatrix}$. Then $2m$ is represented by $\begin{pmatrix} 2 & \frac12 \\ \frac12 & 3 \end{pmatrix}$. Assume that  $2x^2+xy+3y^2=2m$ for some integers $x$ and $y$. Then either $x+y$ or $y$ is even. If $x+y=2z$ for some integer $z$, then $2x^2-5xz+6z^2=m$. Hence $m$ is represented by $\begin{pmatrix} 2 & \frac12 \\ \frac12 & 3 \end{pmatrix}$. If $y=2z$, then $x^2+xz+6z^2=m$. Hence   $m$ is represented by $\begin{pmatrix} 1 & \frac12 \\ \frac12 & 6 \end{pmatrix}$.
 Note that 
$$
\gen\left(\begin{pmatrix} 1 & \frac12 \\ \frac12 & 6 \end{pmatrix}\right)\big/\sim=\left\{ \begin{pmatrix} 1 & \frac12 \\ \frac12 & 6 \end{pmatrix}, \begin{pmatrix} 2 & \frac12 \\ \frac12 & 3 \end{pmatrix} \right\}.
$$
Conversely, assume that $m$ is represented by the genus of $\begin{pmatrix} 1 & \frac12 \\ \frac12 & 6 \end{pmatrix}$. If $m$ is represented by $\begin{pmatrix} 1 & \frac12 \\ \frac12 & 6 \end{pmatrix}$, then we have
$$
\langle 4m \rangle \ra \begin{pmatrix} 4 &2 \\2 & 24 \end{pmatrix} \ra \begin{pmatrix} 4 &1 \\1 & 6 \end{pmatrix}. 
$$
 If $m$ is represented by $\begin{pmatrix} 2 & \frac12 \\ \frac12 & 3 \end{pmatrix}$, then we have
 $$
 \langle 4m \rangle \ra \begin{pmatrix} 8 & 2 \\ 2 & 12 \end{pmatrix} \simeq \begin{pmatrix} 8 & 6 \\ 6 & 16 \end{pmatrix} \ra \begin{pmatrix} 8 & 3 \\ 3 & 4 \end{pmatrix} \simeq \begin{pmatrix} 4 &1 \\1 & 6 \end{pmatrix}. 
$$
Hence $4m$ is represented by \!$\begin{pmatrix} 4 & 1 \\ 1 & 6 \end{pmatrix}$ if and only if $m$ is represented by \!$\gen\left(\!\begin{pmatrix} 1 & \frac12 \\ \frac12 & 6 \end{pmatrix}\!\right)$.

Note that 9 binary $\z$-lattices in \eqref{9lattices} have class number $1$. 
Therefore if $m\equiv 1 \pmod 8$ is prime, and  for any prime $q\in \{3,5,7,11,23,31\}$, $m$ is a quadratic residue modulo $q$, then one may easily check that $m$ is represented by $9$ binary $\z$-lattices in \eqref{9lattices} and by both genera in \eqref{2genera}.
This implies   that $m$ is a recoverable number by Proposition \ref{12lattices}. This completes the proof.
\end{proof}

\begin{rmk}{\rm We checked that any non square integer less than or equal to $35$ is not a recoverable number. }
\end{rmk}


\end{document}